\documentclass[10pt,twoside,a4paper]{article}
\usepackage{amsfonts,amssymb,amsmath,amscd,latexsym,makeidx,theorem}
\author{Luca Martinazzi\thanks{Department of Mathematics, ETH Zurich. E.mail: \texttt{luca@math.ethz.ch}}}
\date{April 29, 2008}
\title{Conformal metrics on $\R{2m}$ with constant $Q$-curvature}

\newtheorem{trm}{Theorem}
\newtheorem{prop}[trm]{Proposition}
\newtheorem{cor}[trm]{Corollary}
\newtheorem{lemma}[trm]{Lemma}
\newtheorem{defin}[trm]{Definition}
\newtheorem{open}{Open Question}

\newcommand{\R}[1]{\mathbb{R}^{#1}}

\newcommand{\ve}{\varepsilon}
\newcommand{\bs}{\backslash}

\newcommand{\M}[1]{\mathcal{#1}}

\newenvironment{proof}{\noindent\emph{Proof.}}{\hfill$\square$\medskip}

\newenvironment{ex}{\medskip\noindent\emph{Example.}}{\medskip}

\DeclareMathOperator{\interior}{int}
\DeclareMathOperator{\loc}{loc}

\DeclareMathOperator*{\Intm}{\int\!\!\!\!\!\! \rule[2.6pt]{6.5pt}{.4pt}}

\begin{document}
\maketitle

\begin{abstract}
We study the conformal metrics on $\R{2m}$ with constant $Q$-curvature $Q\in\R{}$ having finite volume, particularly in the case $Q\leq 0$. We show that when $Q<0$ such metrics exist in $\R{2m}$ if and only if $m>1$. Moreover we study their asymptotic behavior at infinity, in analogy with the case $Q>0$, which we treated in a recent paper. When $Q=0$, we show that such metrics have the form $e^{2p}g_{\R{2m}}$, where $p$ is a polynomial such that $2\leq \deg p\leq 2m-2$ and $\sup_{\R{2m}}p<+\infty$. In dimension $4$, such metrics are exactly the polynomials $p$ of degree $2$ with $\lim_{|x|\to+\infty}p(x)=-\infty$.
\end{abstract}

\section{Introduction and statement of the main theorems}

Given a constant $Q\in\R{}$, we consider the solutions to the equation
\begin{equation}
(-\Delta)^m u=Qe^{2mu} \quad \textrm{on }\R{2m},\label{eq0}
\end{equation}
satisfying
\begin{equation}\label{area}
\alpha:=\frac{1}{|S^{2m}|}\int_{\R{2m}} e^{2mu(x)}dx<+\infty.
\end{equation}
Geometrically, if $u$ solves \eqref{eq0} and \eqref{area}, then the conformal metric $g:=e^{2u}g_{\R{2m}}$ has $Q$-curvature $Q_g^{2m}\equiv  Q$ and volume $\alpha|S^{2m}|$. For the definition of the $Q$-curvature and related remarks, we refer to \cite{mar}. Notice that given a solution $u$ to \eqref{eq0} and $\lambda>0$, the function $v:=u-\frac{1}{2m}\log \lambda$ solves
$$(-\Delta)^m v=\lambda Q e^{2mv} \quad \textrm{in }\R{2m},$$
hence what matters is just the sign of $Q$, and we can assume without loss of generality that $Q\in\{0,\pm(2m-1)!\}$.

Every solution to \eqref{eq0} is smooth. When $Q=0$, that follows from standard elliptic estimates; when $Q\neq 0$ the proof is a bit more subtle, see \cite[Corollary 8]{mar}.

For $Q\geq 0$, some explicit solutions to \eqref{eq0} are known. For instance every polynomial of degree at most $2m-2$ satisfies \eqref{eq0} with $Q=0$, and the function $u(x)=\log\frac{2}{1+|x|^2}$ satisfies \eqref{eq0} with $Q=(2m-1)!$ and $\alpha=1$. This latter solution has the property that $e^{2u}g_{\R{2m}}=(\pi^{-1})^*g_{S^{2m}}$, where $\pi:S^{2m}\to\R{2m}$ is the stereographic projection.

For the negative case, we notice that the function $w(x)=\log\frac{2}{1-|x|^2}$ solves $(-\Delta)^m w=-(2m-1)!e^{2m w}$ on the unit ball $B_1\subset\R{2m}$ (in dimension $2$ this corresponds to the Poincar\'e metric on the disk). However, no explicit entire solution to \eqref{eq0} with $ Q<0$ is known, hence one can ask whether such solutions actually exist. In dimension $2$ ($m=1$) it is easy to see that the answer is negative, but quite surprisingly the situation is different in dimension $4$ and higher and we have:

\begin{trm}\label{exist} Fix $Q<0$. For $m=1$ there is no solution to \eqref{eq0}-\eqref{area}. For every $m\geq 2$, there exist (several) radially symmetric solutions to \eqref{eq0}-\eqref{area}.
\end{trm}

Having now an existence result, we turn to the study of the asymptotic behavior at infinity of solutions to \eqref{eq0}-\eqref{area} when $m\geq 2$, $Q<0$, having in mind applications to concentration-compactness problems in conformal geometry. To this end, given a solution $u$ to \eqref{eq0}-\eqref{area}, we define the auxiliary function
\begin{equation}\label{eqv}
v(x):= -\frac{(2m-1)!}{\gamma_m}\int_{\R{2m}}\log\bigg(\frac{|y|}{|x-y|}\bigg) e^{2m u(y)}dy,
\end{equation}
where $\gamma_m :=\omega_{2m} 2^{2m-2}[(m-1)!]^2$ is characterized by the following property:
$$(-\Delta)^m\Big(\frac{1}{\gamma_m}\log\frac{1}{|x|}\Big)=\delta_0\quad\textrm{in }\R{2m}.$$
Then $(-\Delta)^m v=-(2m-1)!e^{2mu}.$ We prove

\begin{trm}\label{clas1} Let $u$ be a solution of \eqref{eq0}-\eqref{area} with $Q=-(2m-1)!$.
Then
\begin{equation}\label{repr}
u(x)=v(x)+p(x),
\end{equation}
where $p$ is a non-constant polynomial of even degree at most $2m-2$. Moreover there exist a constant $a\neq 0$, an integer $1\leq j\leq m-1$ and a closed set $Z\subset S^{2m-1}$ of Hausdorff dimension at most $2m-2$ such that for every compact subset $K\subset S^{2m-1}\backslash Z$ we have
\begin{eqnarray}
\lim_{t\to+\infty}\Delta^\ell v(t\xi)&=&0,\quad \ell=1,\ldots, m-1, \nonumber\\
v(t\xi)&=&2\alpha\log t+o(\log t),\textrm{ as }t\to+\infty, \nonumber\\
\lim_{t\to+\infty}\Delta^j u(t\xi)&=& a, \label{deltaa}
\end{eqnarray}
for every $\xi\in K$ uniformly in $\xi$. If $m=2$, then $Z=\emptyset$ and $\sup_{\R{2m}} u<+\infty$.
Finally
\begin{equation}\label{Rg}
\liminf_{|x|\to +\infty}R_{g_u}(x)= -\infty,
\end{equation}
where $R_{g_u}$ is the scalar curvature of $g_u:=e^{2u}g_{\R{2m}}$.

\end{trm}

Following the proof of Theorem \ref{exist}, it can be shown that the estimate on the degree of the polynomial is sharp.
Recently J. Wei and D. Ye \cite{WY} showed the existence of solutions to $\Delta^2 u=6e^{4u}$ in $\R{4}$ with $\int_{\R{4}}e^{4u}dx<+\infty$ which are not radially symmetric. It is plausible that also in the negative case non-radially symmetric solutions exist.

For the case $ Q=0$ we have

\begin{trm}\label{zero}
When $Q=0$, any solution to \eqref{eq0}-\eqref{area} is a polynomial $p$ with $2\leq \deg p\leq 2m-2$ and with
$$\sup_{\R{2m}}p<+\infty.$$
In particular in dimension $2$ (case $m=1$), there are no solutions. In dimension $4$ the solutions are exactly the polynomials of degree $2$ with $\lim_{|x|\to\infty}p(x)=-\infty$.
Finally, there exist $1\leq j\leq m-1$ and $a<0$ such that
\begin{equation}\label{aj}
\lim_{|x|\to\infty}\Delta^j p(x)=a.
\end{equation}
\end{trm}

The case when $Q>0$, say $Q=(2m-1)!$, has been exhaustively treated. The problem
\begin{equation}\label{positive}
(-\Delta)^m u=(2m-1)!e^{2mu}\quad\textrm{on }\R{2m},\quad \int_{\R{2m}}e^{2mu}dx<+\infty
\end{equation}
admits standard solutions, i.e. solutions of the form $u(x):=\log\frac{2\lambda}{1+\lambda^2|x-x_0|^2}$, $\lambda>0$, $x_0\in \R{2m}$ that arise from the stereographic projection and the action of the M\"obius group of conformal diffeomorphisms on $S^{2m}$. In dimension $2$ W. Chen and C. Li \cite{CL} showed that every solution to \eqref{positive} is standard. Already in dimension $4$, however, as shown by A. Chang and W. Chen \cite{CC}, \eqref{positive} admits non-standard solutions. In dimension $4$ C-S. Lin \cite{lin} classified all solutions $u$ to \eqref{positive} and gave precise conditions in order for $u$ to be a standard solution in terms of its asymptotic behavior at infinity. 

In arbitrary even dimension, A. Chang and P. Yang \cite{CY} proved that solutions of the form
$$u(x)=\log\frac{2}{1+|x|^2}+\xi(\pi^{-1}(x))$$
are standard, where $\pi:S^{2m}\to\R{2m}$ is the stereographic projection and $\xi$ is a smooth function on $S^{2m}$. J. Wei and X. Xu \cite{WX} showed that any solution $u$ to \eqref{positive} is standard under the weaker assumption that $u(x)=o(|x|^2)$ as $|x|\to\infty$, see also \cite{xu}. We recently treated the general case, see \cite{mar}, generalizing the work of C-S. Lin. In particular we proved a decomposition $u=p+v$ as in Theorem \ref{clas1} and gave various analytic and geometric conditions which are equivalent to $u$ being standard. 

The classification of the solutions to \eqref{positive} has been applied in concentration-compactness problems, see e.g. \cite{LS}, \cite{RS}, \cite{mal}, \cite{MS}, \cite{DR}, \cite{str1}, \cite{str2}, \cite{ndi}.  There is an interesting geometric consequence of Theorems \ref{clas1} and \ref{zero}, with applications in concentration-compactness: In the case of a closed manifold, metrics of equibounded volumes and prescribed $Q$-curvatures \emph{of possibly varying sign} cannot concentrate at points of negative or zero $Q$-curvature.
For instance we shall prove in a forthcoming paper \cite{mar2}

\begin{trm}\label{trmquant} Let $(M,g)$ be a $2m$-dimensional closed Riemannian manifold with Paneitz operator $P^{2m}_g$ satisfying $\ker P^{2m}_g=\{const\}$, and let $u_k:M\to \R{}$ be a sequence of solutions of 
\begin{equation}\label{deltaT}
P^{2m}_g u_k+ Q^{2m}_g=Q_k e^{2m u_k},
\end{equation}
where $Q^{2m}_g$ is the $Q$-curvature of $g$ (see e.g. \cite{cha}), and where the $Q_k$'s are given continuous functions with $Q_k\to Q_0$ in $C^0$. Assume also that there is a $\Lambda>0$ such that
\begin{equation}\label{vkT}
\int_{M}e^{2mu_k}\mathrm{dvol}_g \leq \Lambda,
\end{equation}
for all $k$. Then one of the following is true.
\begin{enumerate}
\item[(i)] For every $0\leq \alpha<1$, a subsequence is converging in $C^{2m-1,\alpha}(M)$. 
\item[(ii)] There exists a finite set $S=\{x^{(i)}:1\leq i\leq I\}$ such that $u_k\to -\infty$ in $L^\infty_{\loc}(M\bs S)$. Moreover
\begin{equation}\label{intQg}
\int_M Q_g \mathrm{dvol}_g = I(2m-1)!|S^{2m}|,
\end{equation}
and
\begin{equation}\label{QgL}
Q_k e^{2mu_k}\mathrm{dvol}_g \rightharpoonup\sum_{i=1}^I (2m-1)!|S^{2m}|\delta_{x^{(i)}},
\end{equation}
in the sense of measures. Finally $Q_0(x^{(i)})>0$ for $1\leq i\leq I$.
\end{enumerate}
\end{trm}

In sharp contrast with Theorem \ref{trmquant}, on an open domain $\Omega\subset \R{2m}$ (or a manifold with boundary), $m>1$, concentration is possible at points of negative or zero curvature. Indeed, take any solution $u$ of \eqref{eq0}-\eqref{area} with $ Q\leq 0$, whose existence is given by Theorem \ref{exist}, and consider the sequence 
$$u_k(x):=u(k(x-x_0))+\log k, \quad \textrm{for }x\in\Omega$$
for some fixed $x_0\in\Omega$. Then $(-\Delta)^m u_k=Q e^{2mu_k}$ and $u_k$ concentrates at $x_0$ in the sense that as $k\to \infty$ we have $u_k(x_0)\to +\infty$, $u_k\to-\infty$ a.e. in $\Omega$ and $e^{2mu_k}dx\rightharpoonup \alpha|S^{2m}|\delta_{x_0}$ in the sense of measures.

The $2$ dimensional case ($m=1$) is different and concentration at points of non-positive curvature can be ruled out on open domains too, because otherwise a standard blowing-up procedure would yield a solution to \eqref{eq0}-\eqref{area} with $Q\leq 0$, contradicting with Theorem \ref{exist}.

An immediate consequence of Theorem \ref{trmquant} and the Gauss-Bonnet-Chern formula, is the following compactness result (see \cite{mar2}):

\begin{cor} In the hypothesis of Theorem \ref{trmquant} assume that either
\begin{enumerate}
\item $\chi(M)\leq 0$ and $\dim M\in\{2,4\}$, or
\item $\chi(M)\leq 0$, $\dim M\geq 6$ and $(M,g)$ is locally conformally flat,
\end{enumerate}
where $\chi(M)$ is the Euler-Poincar\'e characteristic of $M$. Then only case (i) in Theorem \ref{trmquant} occurs.
\end{cor}

\medskip

The paper is organized as follows. The proof of Theorems \ref{exist}, \ref{clas1} and \ref{zero} is given in the following three sections; in the last section we collect some open questions.
In the following, the letter $C$ denotes a generic constant, which may change from line to line and even within the same line.

\section{Proof of Theorem \ref{exist}}

Theorem \ref{exist} follows from Propositions \ref{m1} and \ref{m2} below. 

\begin{prop}\label{m1} For $m=1$, $Q<0$ there are no solutions to \eqref{eq0}-\eqref{area}.
\end{prop}

\begin{proof} Assume that such a solution $u$ exists. Then, by the maximum principle, and Jensen's inequality,
$$\Intm_{\partial B_R}ud\sigma\geq u(0),\qquad \int_{\partial B_R}e^{2u}d\sigma\geq 2\pi Re^{2u(0)}.$$
Integrating in $R$ on $[1,+\infty)$, we get
$$\int_{\R{2}}e^{2u}dx=+\infty,$$
contradiction.
\end{proof}

\begin{lemma}\label{lapl} Let $u(r)$ be a smooth radial function on $\R{n}$, $n\geq 1$. Then there are \emph{positive} constants $b_m$ depending only on $n$ such that
\begin{equation}\label{eqlapl}
\Delta^m u(0)=b_mu^{(2m)}(0),
\end{equation}
$u^{(2m)}:=\frac{\partial^{2m}u}{\partial r^{2m}}$. In particular $\Delta^mu(0)$ has the sign of $u^{(2m)}(0)$.
\end{lemma}

For a proof see \cite{mar}.

\begin{prop}\label{m2} For $m\geq 2$, $Q<0$ there exist radial solutions to \eqref{eq0}-\eqref{area}.
\end{prop}

\begin{proof} We consider separately the cases when $m$ is even and when $m$ is odd. 
\noindent\emph{Case 1: $m$ even.} Let $u=u(r)$ be the unique solution of the following ODE:

$$
\left\{
\begin{array}{ll}
\Delta^m u(r)=-(2m-1)! e^{2mu(r)} &\\
u^{(2j+1)}(0)=0& 0\leq j\leq m-1\\
u^{(2j)}(0)=\alpha_j\leq 0 & 0\leq j\leq m-1,
\end{array}
\right.
$$
where $\alpha_0=0$ and $\alpha_1<0$.
We claim that the solution exists for all $r\geq 0$. To see that, we shall use barriers, compare \cite[Theorem 2]{CC}. Let us define
$$w_+(r)=\frac{\alpha_1}{2} r^2,\quad g_+ :=w_+-u.$$
Then $\Delta^m g_+\geq 0$. By the divergence theorem,
$$\int_{B_R}\Delta^j g_+dx=\int_{\partial B_R}\frac{d \Delta^{j-1} g_+}{dr}d\sigma.$$
Moreover, from Lemma \ref{lapl}, we infer
$$\Delta^j g_+(0)\geq 0 \quad\textrm{ for }0\leq j\leq m-1,$$
hence we see inductively that $\Delta^j g_+(r)\geq 0$ for every $r$ such that $g_+(r)$ is defined and for $0\leq j\leq m-1$. In particular $g_+\geq 0$ as long as it exists.

Let us now define
$$w_-(r):=\sum_{i=0}^{m-1}\beta_i r^{2i}-A\log\frac{2}{1+r^2}, \quad g_-:=u-w_-,$$
where the $\beta_i$'s and $A$ will be chosen later. Notice that
$$\Delta^m w_-(r)=\Delta^m\bigg(-A\log\frac{2}{1+r^2}\bigg)=-(2m-1)! A\bigg(\frac{2}{1+r^2}\bigg)^{2m}.$$
Since $\alpha_1<0$,
$$\lim_{r\to+\infty}\frac{\big(\frac{2}{1+r^2}\big)^{2m}}{e^{m\alpha_1 r^2}}=+\infty,$$
and taking into account that $u\leq w_+$, we can choose $A$ large enough, so that
\begin{eqnarray*}
\Delta^m g_-(r)&=&(2m-1)!\bigg[A \bigg(\frac{2}{1+r^2}\bigg)^{2m}-e^{2mu(r)}\bigg]\\
&\geq& (2m-1)!\bigg[A \bigg(\frac{2}{1+r^2}\bigg)^{2m}-e^{m\alpha_1r^2}\bigg]\geq 0.
\end{eqnarray*}
We now choose each $\beta_i$ so that
$$ \Delta^j g_-(0)\geq 0,\quad 0\leq j\leq m-1,$$
and proceed by induction as above to prove that $g_-\geq 0$. Hence
$$w_-(r)\leq u(r)\leq w_+(r)$$
as long as $u$ exists, and by standard ODE theory, that implies that $u(r)$ exists for all $r\geq 0$. Finally
$$\int_{\R{2m}}e^{2mu(|x|)}dx\leq \int_{\R{2m}}e^{m\alpha_1|x|^2}dx<+\infty.$$

\medskip

\noindent\emph{Case 2: $m\geq 3$ odd.} Let $u=u(r)$ solve

$$
\left\{
\begin{array}{ll}
\Delta^m u(r)=(2m-1)! e^{2mu(r)} &\\
u^{(2j+1)}(0)=0& 0\leq j\leq m-1\\
u^{(2j)}(0)=\alpha_j\leq 0 & 0\leq j\leq m-1,
\end{array}
\right.
$$
where the $\alpha_i$'s have to be chosen. Set
$$w_+(r):=\beta-r^2-\log\frac{2}{1+r^2},\quad g_+:=w_+-u,$$
where $\beta<0$ is such that $e^{-r^2+\beta}\leq \big(\frac{2}{1+r^2}\big)^2$, hence
$$\frac{2}{1+r^2}-\frac{1+r^2}{2}e^{-r^2+\beta}\geq 0\quad \textrm{for all }r>0.$$
Then, as long as $g_+\geq 0$, we have
\begin{eqnarray*}
\Delta^m g_+(r)&=&(2m-1)!\bigg[\bigg(\frac{2}{1+r^2}\bigg)^{2m}-e^{2mu(r)}\bigg]\\
&\geq& (2m-1)!\bigg[\bigg(\frac{2}{1+r^2}\bigg)^{2m}-e^{2mw_+(r)}\bigg] \geq 0
\end{eqnarray*}
Choose now the $\alpha_i$'s so that, $u^{(2i)}(0)< w_+^{(2i)}(0)$, for $0\leq i\leq m-1$. From Lemma \ref{lapl}, we infer that
$$\Delta^{i}g_+(0)\geq 0,\quad 0\leq i\leq m-1,$$
and we see by induction that $g_+\geq 0$ as long as it is defined.
As lower barrier, define
$$w_-(r)=\sum_{i=0}^{m-1}\beta_ir^{2i},\quad g_-:=u-w_-,$$
where the $\beta_i$'s are chosen so that $\Delta^ig_-(0)\geq 0$. Then, observing that
$$\Delta^m g_-(r)=(2m-1)!e^{2mu(r)}>0,$$
as long as $u$ is defined, we conclude as before that $g_-\geq 0$ as long as it is defined. Then $u$ is defined for all times.

Let $R>0$ be such that, for every $r\geq R$, $w_+(r)\leq -\frac{r^2}{2}$. Then
$$\int_{\R{2m}}e^{2m u(|x|)}dx\leq \int_{B_R}e^{2mu(|x|)}dx+\int_{\R{2m}\backslash B_R}e^{-m|x|^2}dx<+\infty.$$
\end{proof}

\section{Proof of Theorem \ref{clas1}}\label{main}

The proof of Theorem \ref{clas1} is divided in several lemmas. The following Liouville-type theorem will prove very useful.

\begin{trm}\label{trmliou2} Consider $h:\R{n}\to \R{}$ with $\Delta^m h=0$ and $h\leq u-v$, where $e^{pu}\in L^1(\R{n})$ for some $p>0$,  $(-v)^+\in L^1(\R{n})$. Then $h$ is a polynomial of degree at most $2m-2$. 
\end{trm}

\begin{proof} As in \cite[Theorem 5]{mar}, for any $x\in\R{2m}$ we have
\begin{eqnarray}
|D^{2m-1}h(x)|&\leq&\frac{C}{R^{2m-1}}\Intm_{B_R(x)}|h(y)|dy \nonumber \\
&=&-\frac{C}{R^{2m-1}}\Intm_{B_R(x)}h(y)dy+\frac{2C}{R^{2m-1}}\Intm_{B_R(x)} h^+dy \label{eqla}
\end{eqnarray}
and
$$\Intm_{B_R(x)}h(y)dy=O(R^{2m-2}),\quad \textrm{as }R\to\infty.$$
Then
$$\Intm_{B_R(x)}h^+dy\leq \Intm_{B_R(x)} u^+dy+C\Intm_{B_R(x)}(-v)^+dy
\leq\frac{1}{p}\Intm_{B_R(x)}e^{pu}dy+\frac{C}{R^{2m}},$$
and both terms in \eqref{eqla} divided by $R^{2m-1}$ go to $0$ as $R\to\infty$.
\end{proof}

\begin{lemma}\label{lemmabeta} Let $u$ be a solution of \eqref{eq0}-\eqref{area}.
Then, for $|x|\geq 4$
\begin{equation}\label{eqlog}
v(x)\leq 2\alpha \log|x|+C.
\end{equation}
\end{lemma}

\begin{proof} As in \cite[Lemma 9]{mar}, changing $v$ with $-v$.
\end{proof}

\begin{lemma}\label{intv+} For any $\ve>0$, there is $R>0$ such that for $|x|\geq R$,
\begin{equation}\label{vR}
v(x)\geq \Big(2\alpha-\frac{\ve}{2}\Big)\log|x|+\frac{(2m-1)!}{\gamma_m}\int_{B_1(x)}\log|x-y|e^{2mu(y)}dy.
\end{equation}
Moreover
\begin{equation}\label{v++}
(-v)^+\in L^1(\R{2m}).
\end{equation}
\end{lemma}

\begin{proof}
To prove \eqref{vR} we follow \cite{lin}, Lemma 2.4.
Choose $R_0>0$ such that
$$\frac{1}{|S^{2m}|}\int_{B_{R_0}}e^{2mu}dx\geq \alpha-\frac{\ve}{16},$$
and decompose
\begin{eqnarray*}
\R{2m}&=&B_{R_0}\cup A_1\cup A_2,\\
A_1&:=&\{y\in \R{2m}: 2|x-y|\leq |x|,|y|\geq R_0\},\\
A_2&:=&\{y\in \R{2m}: 2|x-y|> |x|, |y|\geq R_0\}.
\end{eqnarray*}
Next choose $R\geq 2$ such that for $|x|>R$ and $|y|\leq R_0$, we have
$\log\frac{|x-y|}{|y|}\geq \log|x|-\ve$. Then, observing that $\frac{(2m-1)!|S^{2m}|}{\gamma_m}=2$, we have for $|x|>R$
\begin{eqnarray}
\frac{(2m-1)!}{\gamma_m}\int_{B_{R_0}}\log\frac{|x-y|}{|y|}e^{2mu(y)}dy&\geq &\Big(\log|x|-\frac{\ve}{16}\Big)\frac{(2m-1)!}{\gamma_m}\int_{B_{R_0}}e^{2mu}dy\nonumber\\
&\geq& \Big(2\alpha-\frac{\ve}{8}\Big)\log|x|-C\ve.\label{BR0}
\end{eqnarray}
Observing that $\log|x-y|\geq 0$ for $y\notin B_1(x)$, $\log|y|\leq \log(2|x|)$ for $y\in A_1$,  $\int_{A_1}e^{2mu}dy\leq \frac{\ve|S^{2m}|}{16}$ and $\log(2|x|)\leq 2\log|x|$ for $|x|\geq R$,  we infer
\begin{eqnarray}
\int_{A_1}\log\frac{|x-y|}{|y|}e^{2mu(y)}dy&=&\int_{A_1}\log|x-y|e^{2mu(y)}dy-\int_{A_1}\log|y|e^{2mu(y)}dy\nonumber\\
&\geq&\int_{B_1(x)}\log|x-y|e^{2mu(y)}dy-\log(2|x|)\int_{A_1}e^{2mu}dy\nonumber\\
&\geq&\int_{B_1(x)}\log|x-y|e^{2mu(y)}dy-\log|x|\frac{\ve|S^{2m}|}{8}.\label{AA1}
\end{eqnarray}
Finally, for $y\in A_2$, $|x|>R$ we have that $\frac{|x-y|}{|y|}\geq \frac{1}{4}$, hence
\begin{equation}\label{AA2}
\int_{A_2}\log\frac{|x-y|}{|y|}e^{2mu(y)}dy\geq -\log(4)\int_{A_2}e^{2mu}dy\geq -C\ve.
\end{equation}
Putting together \eqref{BR0}, \eqref{AA1} and \eqref{AA2}, and possibly taking $R$ even larger, we obtain \eqref{vR}. From \eqref{vR} and Fubini's theorem
\begin{eqnarray*}
\int_{\R{2m}\bs B_R}(-v)^+dx&\leq& C\int_{\R{2m}}\int_{\R{2m}}\chi_{|x-y|< 1}\log\frac{1}{|x-y|} e^{2mu(y)}dydx\nonumber\\
&=& C \int_{\R{2m}}e^{2mu(y)}\int_{B_1(y)}\log\frac{1}{|x-y|}dxdy \\
&\leq &C\int_{R^{2m}}e^{2mu(y)}dy<\infty.
\end{eqnarray*}
Since $v\in C^\infty(\R{2m})$, we conclude that $\int_{B_R}(-v)^+dx<\infty$ and \eqref{v++} follows.
\end{proof}

\begin{lemma}\label{Deltapol} Let $u$ be a solution of \eqref{eq0}-\eqref{area}, with $m\geq 2$.
Then $u=v+p$, where $p$ is a polynomial of degree at most $2m-2$. \end{lemma}

\begin{proof} Let $p:=u-v$. Then $\Delta^m p=0$. Apply \eqref{v++} and Theorem \ref{trmliou2}.
\end{proof}

\begin{lemma}\label{polinf} Let $p$ be the polynomial of Lemma \ref{Deltapol}. Then if $m=2$, there exists $\delta>0$ such that
\begin{equation}\label{pdelta}
p(x)\leq -\delta|x|^2+C.
\end{equation}
In particular $\lim_{|x|\to\infty}p(x)=-\infty$ and $\deg p=2$. For $m\geq 3$ there is a (possibly empty) closed set $Z\subset S^{2m-1}$ of Hausdorff dimension $\dim^\M{H}(Z)\leq 2m-2$ such that for every $K\subset S^{2m-1}\backslash Z$ closed, there exists $\delta=\delta(K)>0$ such that 
\begin{equation}\label{pdelta2}
p(x)\leq -\delta|x|^2+C \quad \textrm{for } \frac{x}{|x|}\in K.
\end{equation}
Consequently $\deg p$ is even.
\end{lemma}

\begin{proof} 
From \eqref{v++}, we infer that there is a set $A_0$ of finite measure such that
\begin{equation}\label{A0}
v(x)\geq -C\quad \textrm{in }\R{2m}\backslash A_0.
\end{equation}

\medskip

\noindent\emph{Case $m=2$.} Up to a rotation, we can write 
$$p(x)=\sum_{i=1}^4 (b_ix_i^2+c_ix_i)+b_0.$$
Assume that $b_{i_0}\geq 0$ for some $1\leq i_0\leq 4$. Then on the set
$$A_1:=\{x\in \R{4}:|x_i|\leq 1\textrm{ for }i\neq i_0,\; c_{i_0}x_{i_0}\geq 0\}$$
we have $p(x)\geq - C$. Moreover $|A_1|=+\infty$. Then, from \eqref{A0} we infer
\begin{equation}\label{areainf}
\int_{\R{4}}e^{4u}dx\geq\int_{A_1\backslash A_0}e^{4(v+p)}dx\geq C|A_1\backslash A_0|=+\infty,
\end{equation}
contradicting \eqref{area}. Therefore $b_i<0$ for every $i$ and \eqref{pdelta} follows at once.

\medskip

\noindent\emph{Case $m\geq 3$.} From \eqref{area} and \eqref{A0} we infer that $p$ cannot be constant. Write
$$p(t\xi)=\sum_{i=0}^{d}a_i(\xi)t^i,\qquad d:=\deg p,$$
where for each $0\leq i\leq d$, $a_i$ is a homogeneous polynomial of degree $i$ or $a_i\equiv 0$.
With a computation similar to \eqref{areainf}, \eqref{area} and \eqref{A0} imply that $a_d(\xi)\leq 0$ for each $\xi\in S^{2m-1}$. Moreover $d$ is even, otherwise $a_d(\xi)=-a_d(-\xi)\leq 0$ for every $\xi\in S^{2m-1}$, which would imply $a_d\equiv 0$. Set
$$Z=\{\xi\subset S^{2m-1}:a_d(\xi)=0\}.$$
We claim that $\dim^\M{H}(Z)\leq 2m-2$. To see that, set
$$V:=\{x\in \R{2m}\;:\;a_d(x)=0\}=\{t\xi\;:\; t\geq 0,\; \xi\in Z\}.$$
Since $V$ is a cone and $Z=V\cap S^{2m-1}$, we only need to show that $\dim^\M{H}(V)\leq 2m-1$. Set
$$V_i:=\{x\in \R{2m}\;:\; a_d(x)=\ldots= \nabla^i a_d(x)=0,\; \nabla^{i+1}a_d(x)\neq 0\}.$$
Noticing that $V_i=\emptyset$ for $i\geq d$ (otherwise $a_d\equiv 0$), we find
$V=\cup_{i=0}^{d-1}V_i.$
By the implicit function theorem, $\dim^{\M{H}}(V_i)\leq 2m-1$ for every $i\geq 0$ and the claim is proved.

Finally, for every compact set $K\subset S^{2m-1}\backslash Z$, there is a constant $\delta>0$ such that $a_d(\xi)\leq -\frac{\delta}{2}$, and since $d\geq 2$, \eqref{pdelta2} follows.
\end{proof}

\begin{cor}\label{uC} Any solution $u$ of \eqref{eq0}-\eqref{area} with $m=2$, $Q<0$ is bounded from above.
\end{cor}

\begin{proof} Indeed $u=v+p$ and, for some $\delta>0$,
$$v(x)\leq 2\alpha\log|x|+C,\quad p(x)\leq -\delta|x|^2+C.$$
\end{proof}

\begin{lemma}\label{Deltav} Let $v:\R{2m}\to\R{}$ be defined as in \eqref{eqv} and $Z$ as in Lemma \ref{polinf}. Then for every $K\subset S^{2m-1}\backslash Z$ compact we have
\begin{equation}\label{dv}
\lim_{t\to+\infty}\Delta^{m-j} v(t\xi)=0, \quad j=1,\ldots,m-1
\end{equation}
for every $\xi\in K$ uniformly in $\xi$; for every $\ve>0$ there is $R=R(\ve,K)>0$ such that, for $t>R$, $\xi\in K$,
\begin{equation}\label{veps}
v(t\xi)\geq (2\alpha -\varepsilon)\log t
\end{equation}
\end{lemma}

\begin{proof} Fix $K\in S^{2m-1}\backslash Z$ compact and set
$\M{C}_K:=\{t\xi:t\geq 0, \xi\in K\}.$
For any $\sigma>0$, $1\leq j\leq 2m-1$,
\begin{equation}\label{leb}
\int_{\R{2m}\backslash B_\sigma(x)}\frac{e^{2mu(y)}}{|x-y|^{2j}}dy\to 0\quad\textrm{as }|x|\to\infty
\end{equation}
by dominated convergence. Choose a compact set $\widetilde K\subset S^{2m-1}\backslash Z$ such that $K\subset \interior(\widetilde K)\subset S^{2m-1}$. Since $u\leq C(\widetilde K)$ on $\M{C}_{\widetilde K}$ by Lemma \ref{lemmabeta} and Lemma \ref{polinf}, we can choose $\sigma=\sigma(\ve)>0$ so small that
$$\int_{B_\sigma(x)}\frac{e^{2mu}}{|x-y|^{2j}}dy\leq C(\widetilde K)\int_{B_\sigma(x)}\frac{1}{|x-y|^{2j}}dy\leq C(\widetilde K)\ve, \quad \textrm{for }x\in \M{C}_K, \;|x|\textrm{ large},$$
where $|x|$ is so large that $B_\sigma(x)\subset \M{C}_{\widetilde K}$.
Therefore
$$(-1)^{j+1}\Delta^j v(x)=C\int_{\R{2m}}\frac{e^{2mu}}{|x-y|^{2j}}dy\to 0,\quad \textrm{for }x\in\M{C}_K,\textrm{ as } |x|\to\infty,$$

We have seen in Lemma \ref{intv+}, that for any $\ve>0$ there is $R>0$ such that for $|x|\geq R$
\begin{equation}\label{211}
v(x)\geq \Big(2\alpha-\frac{\ve}{2}\Big)\log|x| +\frac{(2m-1)!}{\gamma_m}\int_{B_1(x)}\log |x-y|e^{2m u(y)}dy,
\end{equation}
and \eqref{veps} follows easily by choosing $\widetilde K$ as above and observing that $u\leq C(\widetilde K)$ on $\M{C}_{\widetilde K}$, hence on $B_1(x)$ for $x\in\M{C}_K$ with $|x|$ large enough.
\end{proof}

\medskip

\noindent\emph{Proof of Theorem \ref{clas1}.} The decomposition $u=v+p$ and the properties of $v$ and $p$ follow at once from Lemmas \ref{lemmabeta}, \ref{Deltapol}, \ref{polinf} and \ref{Deltav}; \eqref{Rg} follow as in \cite[Theorem 2]{mar}.
As for \eqref{deltaa}, let $j$ be the largest integer such that $\Delta^j p\not\equiv 0$. Then $\Delta^{j+1}p\equiv 0$ and from Theorem \ref{trmliou2} we infer that $\deg p=2j$, hence $\Delta^j p\equiv a\neq 0$.
\hfill $\square$

\section{The case $Q=0$}

\noindent\emph{Proof of Theorem \ref{zero}.} From Theorem \ref{trmliou2}, with $v\equiv 0$, we have that $u$ is a polynomial of degree at most $2m-2$. Then, as in \cite[Lemma 11]{mar}, we have 
$$\sup_{\R{2m}}u<+\infty,$$
and, since $u$ cannot be constant, we infer that $\deg u\geq 2$ is even. 
The proof of \eqref{aj} is analogous to the case $Q<0$, as long as we do not care about the sign of $a$. To show that $a<0$, one proceeds as in \cite[Theorem 2]{mar}.
For the case $m=2$ one proceeds as in Lemma \ref{polinf}, setting $v\equiv 0$ and $A_0=\emptyset$.
\hfill$\square$

\begin{ex} One might believe that every polynomial $p$ on $\R{2m}$ of degree at most $2m-2$ with $\int_{\R{2m}}e^{2mp}dx<\infty$ satisfies $\lim_{|x|\to\infty}p(x)=-\infty$, as in the case $m=2$.
Consider on $\R{2m}$, $m\geq 3$ the polynomial $u(x)=-(1+x_1^2)|\widetilde x|^2$, where $\widetilde x=(x_2,\ldots,x_{2m})$. Then $\Delta^m u\equiv 0$ and
\begin{eqnarray*}
\int_{\R{2m}}e^{2mu}dx&=&\int_{\R{}}\int_{\R{2m-1}}e^{-2m(1+x_1^2)|\widetilde x|^2}d\widetilde xdx_1\\
&=&\int_{\R{}}\frac{dx_1}{(1+x_1^2)^\frac{2m-1}{2}}\cdot\int_{\R{2m-1}}e^{-2m|\widetilde y|^2}d\widetilde y<+\infty.
\end{eqnarray*}
On the other hand, $\limsup_{|x|\to\infty}u(x)=0$.
\end{ex}

\section{Open questions}

\begin{open}
Does the claim of Corollary \ref{uC} hold for $m>2$? In other words, is any solution $u$ to \eqref{eq0}-\eqref{area} with $Q<0$ bounded from above?
\end{open}

This is an important regularity issue, in particular with regard to the behavior at infinity of the function $v$ defined in \eqref{eqv}. If $\sup_{\R{2m}} u<+\infty$, then one can take $Z=\emptyset$ in Theorem \ref{clas1}, as in the case $Q>0$, see \cite[Theorem 1]{mar}.

\begin{defin} Let $\M{P}_0^{2m}$ be the set of polynomials $p$ of degree at most $2m-2$ on $\R{2m}$ such that $e^{2mp}\in L^1(\R{2m})$. Let $\M{P}_+^{2m}$ be the set of polynomials $p$ of degree at most $2m-2$ on $\R{2m}$ such that there exists a solution $u=v+p$ to \eqref{eq0}-\eqref{area} with $Q>0$. Similarly for $\M{P}_-^{2m}$ with $Q<0$.
\end{defin}

Related to the first question is the following

\begin{open} What are the sets $\M{P}^{2m}_0$, $\M{P}^{2m}_{\pm}$? Is it true that $\M{P}_0^{2m}\subset \M{P}_+^{2m}$ and $\M{P}_0^{2m}\subset\M{P}_-^{2m}$?
\end{open}

J. Wei and D. Ye \cite{WY} proved that $\M{P}_0^4\subset \M{P}_+^4$ (and actually more). Consider now on $\R{2m}$, $m\geq 3$, the polynomial
$$p(x)=-(1+x_1^2)|\widetilde x|^2, \quad\widetilde x=(x_2,\ldots, x_{2m}).$$
As seen above, $e^{2mp}\in L^1(\R{2m})$, hence $p\in\M{P}^{2m}_0$. Assume that $p\in \M{P}^{2m}_-$ as well, i.e. there is a function $u=v+p$ satisfying \eqref{eq0}-\eqref{area} and $Q<0$. Then we claim that $\sup_{\R{2m}} u=\infty$. Assume by contradiction that $u$ is bounded from above. Then \eqref{eqlog} and \eqref{vR} imply that
$$v(x)=2\alpha \log |x|+o(\log |x|),\quad \textrm{as }|x|\to\infty.$$
Therefore,
$$\lim_{x_1\to \infty}u(x_1,0,\ldots,0)=\lim_{x_1\to\infty}2\alpha\log x_1=\infty,$$
contradiction.

\begin{open} Even in the case that $u$ is not bounded from above, is it true that one can take $Z=\emptyset$ in Theorem \ref{clas1} for $m\geq 3$ also?
\end{open}
For instance, in order to show that $v(x)=2\alpha\log|x|+o(\log|x|)$ as $|x|\to+\infty$, thanks to \eqref{vR}, it is enough to show that
$$\int_{B_1(x)}\log|x-y|e^{2mu(y)}dy=o(\log|x|),\quad \textrm{as }|x|\to +\infty,$$
which is true if $\sup_{\R{2m}} u<\infty$, but it might also be true if $\sup_{\R{2m}}u=\infty$.

\begin{open} What values can the $\alpha$ given by \eqref{eq0}-\eqref{area} assume for a fixed $Q$?
\end{open}
As usual, it is enough to consider $Q\in \{0,\pm(2m-1)!\}$. When $m=1$, $Q=1$, then $\alpha =1$, see \cite{CL}. When $m=2$, $Q=6$, then $\alpha$ can take any value in $(0,1]$, as shown in \cite{CC}. Moreover $\alpha$ cannot be greater than $1$ and the case $\alpha=1$ corresponds to standard solutions, as proved in \cite{lin}.  For the trivial case $Q=0$, $\alpha$ can take any positive value, and for the other cases we have no answer.


\end{document}